\documentclass[a4paper]{amsart}

\usepackage{amsmath, amssymb, amscd, amsthm, a4wide,verbatim,enumerate}
\usepackage{hyperref}

\newtheorem{thm}{Theorem}[section]
\newtheorem{cor}[thm]{Corollary}
\newtheorem{lemma}[thm]{Lemma}
\newtheorem{prop}[thm]{Proposition}

\theoremstyle{definition}
\newtheorem{remark}[thm]{Remark}
\numberwithin{equation}{section}

\newcommand{\R}{\mathbb R}
\newcommand{\N}{\mathbb N}
\newcommand{\C}{\mathbb C}
\newcommand{\Z}{\mathbb Z}

\newcommand{\al}{\alpha}

\newcommand{\ga}{\gamma}

\newcommand{\de}{\delta}

\newcommand{\eps}{\varepsilon}
\newcommand{\si}{\sigma}
\newcommand{\te}{\theta}

\newcommand{\rphis}[5]{\,_{#1}\varphi_{#2}\!\left( \genfrac{.}{.}{0pt}{}{#3}{#4}
\,;#5 \right)}

\newcommand{\Res}[1]{\underset{#1}{\mathrm{Res}}}

\newcommand{\mhyphen}{\text{--}}

\begin{document}
\title{Orthogonality relations for Al-Salam--Carlitz polynomials of type II}
\author{Wolter Groenevelt}
\address{Technische Universiteit Delft, DIAM, PO Box 5031,
2600 GA Delft, the Netherlands}
\email{w.g.m.groenevelt@tudelft.nl}

\begin{abstract}
Using a special case of Askey's $q$-beta integral evaluation formula, we determine orthogonality relations for the Al-Salam--Carlitz polynomials of type II with respect to a family of measures supported on a discrete subset of $\mathbb R$. From spectral analysis of the corresponding second-order $q$-difference operator we obtain an infinite set of functions that complement the Al-Salam--Carlitz II polynomials to an orthogonal basis of the associated $L^2$-space.
\end{abstract}

\maketitle

\section{Introduction}
To an explicit evaluation of a beta-type integral one can often associate orthogonal polynomials. The main example of this is, of course, Euler's beta integral, which has Jacobi polynomials as corresponding orthogonal polynomials. In this paper we consider orthogonal polynomials corresponding to a special case of Askey's $q$-beta integral \cite{Ask}
\begin{equation} \label{eq:Askey-beta-integral}
\int_{\R_q} \frac{ (ax,bx;q)_\infty }{ (cx,dx;q)_\infty } d_q x = (1-q)z_+\frac{  (q,a/c,a/d,b/c,b/d;q)_\infty \te(z_-/z_+, cdz_-z_+;q) }{ (ab/cdq;q)_\infty \te(cz_-, dz_-, cz_+, dz_+;q)_\infty },
\end{equation}
where $|ab|<|cdq|$. Let us first explain some notations.

Throughout the paper we assume $q\in (0,1)$, and we use notations for $q$-shifted factorials, theta-functions and $q$-hypergeometric functions as in \cite{GR}. Furthermore, for parameters $z_-<0$ and $z_+>0$ we set
\[
\R_q = \R_q(z_-,z_+) = z_-q^\Z \cup z_+ q^\Z,
\]
which we consider as a $q$-analog of the real line. The Jackson $q$-integral over $\R_q$ is defined by
\[
\int_{\R_q} f(x)\, d_q x = (1-q) \sum_{k=-\infty}^\infty\Big( f(z_+ q^k) z_+q^k  - f(z_-q^k) z_- q^k \Big),
\]
for any function $f$ on $\R_q$ for which the sum converges absolutely.
The evaluation of the $q$-beta integral \eqref{eq:Askey-beta-integral} is equivalent to a summation formula involving two $_2\psi_2$-functions, see \cite[Exercise 5.10]{GR} (which contains a misprint).

Assume $\frac{ (ax,bx;q)_\infty }{ (cx,dx;q)_\infty }>0$ for all $x \in \R_q$. The discrete measure in the $q$-beta integral above has only finitely many moments, so there are only finitely many corresponding orthogonal polynomials, which are big $q$-Jacobi polynomials, see \cite{Groen}. The orthogonal polynomials corresponding to the special case $a=0$ are $q$-Meixner polynomials, see \cite{GroenK}, but the support of the measure has to be restricted to $[-qb,\infty)$ (assuming $b>0$). In this paper we consider the special case with $a=b=0$ of \eqref{eq:Askey-beta-integral};
\begin{equation} \label{eq:beta-integral}
\int_{\R_q} \frac{1}{(cx,dx;q)_\infty} d_qx = (1-q)z_+\frac{(q;q)_\infty \te(z_-/z_+,cdz_-z_+;q) }{\te(cz_-,dz_-,cz_+,dz_+;q)}.
\end{equation}
In this case the discrete measure has infinitely many moments, so there exists a set $\{P_n\}_{n\in \N}$ of corresponding orthogonal polynomials.

It is not difficult to determine explicitly the polynomials $P_n$. We denote the right hand side of \eqref{eq:beta-integral} by $B(c,d;z_-,z_+)$. The natural `monomials' in this case are $(cx;q)_m$ and $(dx;q)_m$, and the `moments' corresponding to $(cx;q)_k (dx;q)_m$ are
\[
B(cq^k,dq^m;z_-,z_+) = \left(-\frac{cq^{-m}}{d}\right)^k q^{\frac12k(k-1)} B(c,dq^m;z_-,z_+),
\]
where the latter expression follows from the $\te$-product identity
\begin{equation} \label{eq:te-product}
\te(xq^k;q) = (-x)^{-k} q^{-\frac12 k(k-1)} \te(x;q).
\end{equation}
An application of the $q$-binomial theorem \cite[(II.4)]{GR} then gives us
\[
\begin{split}
\int_{\R_q}  &\rphis{2}{0}{q^{-n},cx}{\mhyphen}{q,\frac{dq^n}{c}} (dx;q)_m \frac{1}{(cx,dx;q)_\infty} d_qx \\
&= \sum_{k=0}^n B(cq^k,dq^m;z_-,z_+) \frac{ (q^{-n};q)_k }{(q;q)_k } \left(-\frac{dq^n}{c}\right)^k q^{-\frac12 k(k-1)} \\
& = B(c,dq^m;z_-,z_+) \sum_{k=0}^n \frac{ (q^{-n};q)_k }{(q;q)_k } q^{(n-m)k} \\
&= B(c,dq^m;z_-,z_+) (q^{-m};q)_n,
\end{split}
\]
from which we see that$\rphis{2}{0}{q^{-n},cx}{\mhyphen}{q,\frac{dq^n}{c}}$, which is a polynomials in $x$ of degree $n$, is orthogonal to all polynomials of degree lower than $n$. A comparison with the orthogonal polynomials in the $q$-Askey-scheme \cite{KLS} shows that the polynomial is an Al-Salam--Carlitz polynomials of type II. These polynomials, introduced by Al-Salam and Carlitz in \cite{ASC}, are given by
\begin{equation} \label{eq:ASCIIpol}
V_n^{(a)}(z;q) = (-a)^n q^{-\frac12 n(n-1)} \rphis{2}{0}{q^{-n},z}{-}{q,\frac{q^n}{a}}.
\end{equation}
We set
\begin{equation} \label{eq:Pn}
P_n(x) = d^n V_n^{(c/d)}(cx;q) = (-c)^n q^{-\frac12 n(n-1)} \rphis{2}{0}{q^{-n},cx}{-}{q,\frac{dq^n}{c}},
\qquad n \in \N,
\end{equation}
then $P_n$ is symmetric in $c$ and $d$ by a limit case of one of Heine's $_2\varphi_1$-transformations (see also Remark \ref{rem:c<->d} later on). Using the symmetry in $c$ and $d$, the orthogonality relations we obtained above are equivalent to the following relations for $P_n$.
\begin{thm} \label{thm:orthogonalitypolynomials}
For $n,m \in \N$,
\begin{equation} \label{eq:orthonalityP}
\begin{split}
\int_{\R_q} P_m(x) P_n(x) \frac{1}{(cx,dx;q)_\infty }\,d_q x &=  \delta_{mn} B(c,d;z_-,z_+) (q;q)_n (cd)^n q^{-n^2},
\end{split}
\end{equation}
with
\[
B(c,d;z_-,z_+)=(1-q)z_+\frac{(q;q)_\infty \te(z_-/z_+,cdz_-z_+;q) }{\te(cz_-,dz_-,cz_+,dz_+;q)}.
\]
\end{thm}
This theorem contains as a special case orthogonality relations for discrete $q$-Hermite polynomials of type II, see Remark \ref{rem:qHermite} later on.

It is well known, see e.g.~\cite{BV,Ch}, that the Al-Salam--Carlitz II polynomials correspond to an indeterminate moment problem. So we just obtained a family, labeled by $z_-$ and $z_+$, of solutions to this moment problem. It turns out the polynomials are not dense in the $L^2$-space associated to \eqref{eq:beta-integral}, so the solutions we just found are not $N$-extremal. It remains then to find a set of functions to complete $\{P_n\}_{n \in \N}$ to an orthogonal basis. We will determine these complementary functions using spectral analysis of the second order $q$-difference operator of which the Al-Salam--Carlitz II polynomials are eigenfunctions. This method, i.e., studying a specific moment problem using spectral analysis of the $q$-difference operator corresponding to the orthogonal polynomials, has successfully been applied in e.g.~\cite{CiccKK},\cite{ChriK},\cite{GroenK}.

Our main motivation for studying the polynomials $P_n$ and the complementary functions (they are denoted by $Q_n$ later on) comes from representation theory of the $\mathrm{SU}_q(2)$ quantum group. In \cite{Groen2} we compute coupling coefficients between two eigenvectors of a special element $\rho_{\tau,\si} \in \mathrm{SU}_q(2)$, which may be considered as a sort of Casimir element, in certain infinite dimensional representations of $\mathrm{SU}_q(2)$. The coupling coefficients turn out to be $2\times2$-matrix-valued orthogonal functions. The functions $P_n$ and $Q_n$ appear in this setting as matrix coefficients of the coupling coefficients.

The outline of the paper is the following. In Section \ref{sect:1phi1} we give a few transformation formulas for $_1\varphi_1$-functions that we need later on in the paper. In Section \ref{sect:orthogonality on R} we perform the spectral analysis of the second order $q$-difference operator $L$ for the Al-Salam--Carlitz II polynomials. In \S\ref{ssec:L} we give the precise definition of $L$, the Hilbert space $\mathcal H$ it is defined on as an unbounded operator, and we show that $L$ extends to a densely defined self-adjoint operator. In \S\ref{ssec:eigenfunctions} we obtain sufficiently many eigenfunctions of $L$, which are given in terms of $_1\varphi_1$-functions. With the eigenfunctions we determine in \S\ref{sect:spectraldecomposition} the spectral decomposition of $L$, which in \S\ref{ssec:orthogonality} leads to an orthogonal basis for $\mathcal H$ in terms of the polynomials $P_n$ and complementary functions $Q_n$, see Theorem \ref{thm:orthogonalityPQ}. Finally, in Section \ref{sect:classicorthogonality} we show that the well-known orthogonality relations for the Al-Salam--Carlitz II polynomials from \cite{ASC} can, in a certain sense, be considered as limit case of \eqref{eq:orthonalityP}.

\section{The $_1\varphi_1$-function} \label{sect:1phi1}
In this paper we mainly use the confluent $q$-hypergeometric function $_{1}\varphi_{1}(a;b;q,z)$, but we will see that this function also appears as a $_2\varphi_1$-function or as a $_2\varphi_0$-function. The function
\[
(b;q)_\infty \rphis{1}{1}{a}{b}{q,z} = \sum_{n=0}^\infty  (a;q)_n (bq^n;q)_\infty (-1)^n q^{\frac12n(n-1)} z^n,
\]
is an entire function in $a$, $b$, and $z$. We collect a few transformation formulas for the $_1\varphi_1$-function. All formulas can be obtained from applying well-known transformation formulas for  $_{2}\varphi_{1}(A,B;C;q,Z)$ with one of the parameters equal to zero.
\begin{lemma} \label{lem:1phi1-transforms}
The following transformation formulas hold:
\begin{align}
\rphis{1}{1}{a}{b}{q,z} & = \frac{ (b/a;q)_\infty }{(b;q)_\infty } \rphis{2}{1}{a,az/b}{0}{q,\frac{b}{a}} \label{eq:transform1}\\
& = (az/b;q)_\infty \rphis{2}{1}{0,b/a}{b}{q,\frac{az}{b}} \label{eq:transform2}\\
& = \frac{(a,z;q)_\infty}{(b;q)_\infty} \rphis{2}{1}{0,b/a}{z}{q,a} \label{eq:transform3}\\
& = \frac{(z;q)_\infty}{(b;q)_\infty} \rphis{1}{1}{az/b}{z}{q,b}. \label{eq:transform4}
\end{align}
\end{lemma}
\begin{proof}
Formula \eqref{eq:transform1} follows from \cite[(III.4)]{GR} with $(A,B,C,Z) = (a,az/c,0,c/a)$. To the $_2\varphi_1$-function we obtain in this way we apply \cite[(III.1)]{GR} with $(A,B,C,Z) = (a,az/c,0,c/a)$ and $(A,B,C,Z) = (az/c,a,0,c/a)$ to find \eqref{eq:transform2} and \eqref{eq:transform3}, respectively. Combining \eqref{eq:transform2} and $\eqref{eq:transform3}$ gives \eqref{eq:transform4}.
\end{proof}
We also need a three-term transformation.
\begin{lemma}
The following three-term transformation formula holds:
\begin{multline} \label{eq:transform5}
\frac{(az;q)_\infty}{(z;q)_\infty}\rphis{1}{1}{a}{az}{q,bz} =\\ \frac{ (b;q)_\infty \te(az;q) }{(b/a;q)_\infty \te(z;q) } \rphis{1}{1}{a}{aq/b}{q,\frac{q^2}{bz}} + \frac{ (a;q)_\infty \te(bz;q) }{(a/b;q)_\infty \te(z;q) } \rphis{1}{1}{b}{bq/a}{q,\frac{q^2}{az}},
\end{multline}
\end{lemma}
\begin{proof}
For \eqref{eq:transform5} we use \cite[(III.32)]{GR}, where we apply Heine's transformation \cite[(III.3)]{GR} to both $_2\varphi_1$-functions on the right hand side;
\[
\begin{split}
\rphis{2}{1}{A,B}{C}{q,Z} = & \frac{ (B, C/A;q)_\infty \te(AZ;q) }{(C,B/A,Z,Cq/ABZ;q)_\infty} \rphis{2}{1}{q/B, C/B}{Aq/B}{q,\frac{q}{Z}} \\
&+ \frac{ (A, C/B;q)_\infty \te(BZ;q) }{(C,A/B,Z,Cq/ABZ;q)_\infty} \rphis{2}{1}{q/A, C/A}{Bq/A}{q,\frac{q}{Z}}.
\end{split}
\]
Now we set $(A,B,C,Z) = (a,b,0,z)$, we apply \eqref{eq:transform2} to the two $_2\varphi_1$-functions on the right hand side, and \eqref{eq:transform1} on the left hand side.
\end{proof}
Finally, we need a transformation involving a terminating series.
\begin{lemma}
For $n \in \N$, the following transformation formula holds:
\begin{equation} \label{eq:transform6}
\rphis{1}{1}{a}{aq^{-n}}{q,bq^{-n}} = \frac{ (b;q)_\infty }{(q/a;q)_n} \left(\frac{b}{a}\right)^n \rphis{2}{0}{q^{-n},q/a}{-}{q,\frac{aq^n}{b}},
\end{equation}
\end{lemma}
\begin{proof}
First we apply \eqref{eq:transform2}, then the $_1\varphi_1$-series becomes a terminating $_2\varphi_1$-series. Reversing the order of summation gives the result.
\end{proof}

\begin{remark} \label{rem:c<->d}
If we first apply \eqref{eq:transform6}, then \eqref{eq:transform4}, and finally \eqref{eq:transform6} again to the $_2\varphi_0$-function in \eqref{eq:Pn}, we obtain
\[
\rphis{2}{0}{q^{-n},cx}{-}{q,\frac{dq^n}{c}} = \left( \frac{d}{c}\right)^n \rphis{2}{0}{q^{-n},dx}{-}{q,\frac{cq^n}{d}}.
\]
So we see that $P_n$ defined by \eqref{eq:Pn} is indeed symmetric in $c$ and $d$. This can also be obtained at once by applying a limit case of Heine's transformation \cite[(III.2)]{GR}.
\end{remark}

\section{Orthogonality relations on $\R$} \label{sect:orthogonality on R}
In this section we obtain orthogonality relations for certain $_1\varphi_1$-functions from spectral analysis of the second-order $q$-difference operator corresponding to the Al-Salam--Carlitz II polynomials.

\subsection{The second-order $q$-difference operator} \label{ssec:L}
Recall $z_-<0$ and $z_+>0$. We set
\[
\R_q^+ = z_+ q^\Z, \qquad \R_q^- = z_- q^\Z, \qquad \R_q = \R_q^- \cup \R_q^+.
\]
By $F(\R_q)$ we denote the vector space consisting of complex-valued functions on $\R_q$, and we define $F(\R_q^+)$ and $F(\R_q^-)$ in the same way.

For nonzero complex parameters $c$ and $d$ we define the second-order $q$-difference operator \mbox{$L=L_{c,d;q}:F(\R_q) \to F(\R_q)$} by
\begin{equation} \label{eq:L}
(Lf)(x) = A(x)\Big( f(qx)-f(x) \Big) + B(x) \Big( f(x/q) - f(x) \Big), \qquad f \in F(\R_q),\ x \in \R_q,
\end{equation}
where
\[
A(x) = \left( 1- \frac{1}{cx} \right) \left( 1- \frac{1}{dx} \right), \qquad B(x) = \frac{q}{cd x^2}.
\]
Clearly, $L$ is symmetric in $c$ and $d$.

We will consider $L$ as an unbounded operator on a Hilbert space $\mathcal H$, that we now introduce.
We define a weight function $w$ on $\R_q$ by
\begin{equation} \label{eq:def w}
w(x) = w(x;c,d;q) = \frac{1}{(cx,dx;q)_\infty}.
\end{equation}
Observe that
\[
w(z_\pm q^k) = 1 + \mathcal O(q^k), \qquad k \to \infty,
\]
and, using the $\te$-product identity \eqref{eq:te-product},
\begin{equation} \label{eq:asymptotics w}
w(z_{\pm} q^k) = \frac{(cd z_\pm^2)^k q^{k(k-1)} }{\te(cz_\pm,dz_\pm;q)} \Big(1+\mathcal O(q^{-k})\Big), \qquad k \to -\infty.
\end{equation}
To ensure positivity of $w$ we assume from here on that the parameters $c$ and $d$ satisfy one of the following conditions:
\begin{itemize}
\item $c \in \C \setminus \R$ and $c=\overline{d}$;
\item $c,d>0$ and there exists a $k \in \Z$ such that $q^{k} < z_+c < q^{k-1}$ and $q^{k} < z_+d < q^{k-1}$;
\item $c,d<0$ and there exists a $k \in \Z$ such that $q^k < z_- c < q^{k-1}$ and $q^k < z_- d < q^{k-1}$.
\end{itemize}
We define $\mathcal H= \mathcal H(c,d;z_-,z_+)$ to be the Hilbert space consisting of functions in $F(\R_q)$ that have finite norm with respect to the inner product
\[
\langle f,g \rangle = \int_{\R_q} f(x) \overline{g(x)} w(x) \,d_q x,
\]
where (recall)
\[
\int_{\R_q} f(x)\, d_q x = (1-q) \sum_{k=-\infty}^\infty \Big( f(z_+ q^k) z_+q^k  - f(z_-q^k) z_- q^k \Big).
\]
We will use a truncated version of the inner product on $\mathcal H$ to show that $L$, with a suitable dense domain, is self-adjoint. For $k,l,m,n \in \Z$ with $l<k$ and $n<m$ we define
\[
\langle f,g \rangle_{k,l;m,n} = \int_{z_+ q^{m+1}}^{ z_+ q^n }  f(x) \overline{g(x)} w(x) \,d_q x- \int_{z_- q^{k+1} }^{z_- q^l
} f(x) \overline{g(x)} w(x) \,d_q x,
\]
where, for $n,m \in \Z$ with $n\geq m$,
\[
\int_{\al q^{n+1}}^{\al q^m} f(x) \,d_q x = (1-q) \sum_{k=m}^n f(\al q^k) \al q^k.
\]
For $f,g \in \mathcal H$ we obtain back the inner product $\langle f,g\rangle$ by letting $k,m \to \infty$ and $l,n \to -\infty$. For $f,g \in F(\R_q)$ we now introduce the Casorati determinant
\begin{gather}
D(f,g)(x) = \Big( f(x) g(qx) - f(qx) g(x) \Big) v(x), \qquad x \in \R_q \label{eq:Casorati determinant}\\
v(x) = \frac{ 1-q}{cdx} \frac{ 1}{(cqx,dqx;q)_\infty}. \nonumber
\end{gather}
Note that $v(x) = (1-q) x A(x) w(x) = (1-q) qx B(qx) w(qx)$.
\begin{lemma} \label{lem:Lfg <-> Dfg}
For $f,g \in F(\R_q)$ we have
\[
\langle Lf,g\rangle_{k,l;m,n} - \langle f,Lg \rangle_{k,l;m,n} = D(f,\overline g) (z_- q^l) - D(f,\overline g) (z_- q^{k-1}) - D(f,\overline g) (z_+ q^m) + D(f,\overline g) (z_+ q^{n-1}).
\]
\end{lemma}
\begin{proof}
We have
\[
\big((Lf)(x)g(x) - f(x) (Lg)(x)\big) (1-q) x w(x) = D(f,\overline{g})(x/q) - D(f,\overline{g})(x),
\]
which follows from a direct verification. As a result $\langle Lf,g\rangle_{k,l;m,n} - \langle f,Lg \rangle_{k,l;m,n}$ are two finite telescoping sums, which proves the lemma.
\end{proof}
From Lemma \ref{lem:Lfg <-> Dfg} we see that the behavior of the Casorati determinant $D(f,g)(x)$, $f,g \in \mathcal H$, at $0$ and $\pm \infty$ is important for determining a dense domain for $L$.
For the Casorati determinant at $\pm \infty$ we need the behavior of $v$ at $\pm\infty$;
\begin{equation} \label{eq:v(infty)}
v(z_\pm q^k) = \frac{(1-q)z_\pm}{\te(cz_\pm,dz_\pm;q)} (cdqz^2_\pm)^k q^{k(k-1)} \Big(1+\mathcal O(q^{-k})\Big), \qquad k \to -\infty.
\end{equation}
This follows from the $\te$-product identity \eqref{eq:te-product}. This leads to the following result.
\begin{lemma} \label{lem:D(f,g)(infty)}
Let $f,g \in \mathcal H$, then $\lim_{k \to -\infty} D(f,g)(z_\pm q^k) =0$.
\end{lemma}
\begin{proof}
From the asymptotic behavior \eqref{eq:asymptotics w} of the weight function $w$ at $\pm \infty$, we see that any $f \in \mathcal H$ satisfies
\[
\lim_{k \to -\infty} (qcdz^2)^{\frac12k} q^{\frac12k(k-1)} f(z_\pm q^k) = 0.
\]
Then for $f,g \in \mathcal H$ we find, using \eqref{eq:v(infty)},
\[
\lim_{k \to -\infty} f(z_\pm q^k) g(z_\pm q^{k+1}) v(z\pm q^k) = 0,
\]
hence also $\lim_{k \to -\infty} D(f,g)(z_\pm q^k) =0$.
\end{proof}

For $f \in F(\R_q)$ we define
\[
f(0^\pm) = \lim_{k \to \infty} f(z_\pm q^k), \qquad f'(0^\pm) = \lim_{k \to \infty} (D_q f)(z_\pm q^k),
\]
provided the limits exists. Here $D_qf$ denotes the $q$-derivative of $f$; $(D_q f)(x)= \frac{ f(x) - f(qx) }{x(1-q)}$ for $x \neq 0$. Let $\mathcal D \subset \mathcal H$ be given by
\[
\mathcal D = \left\{ f \in \mathcal H \mid Lf \in \mathcal H,\ f(0^+)=f(0^-), \ f'(0^+)=f'(0^-) \right\}.
\]
Note that $\mathcal D$ contains the finitely supported functions in $\mathcal H$, hence $\mathcal D$ is dense in $\mathcal H$.
\begin{prop}
The densely defined operator $(L,\mathcal D)$ is self-adjoint.
\end{prop}
This is proved similar as in \cite[Proposition 2.7]{KS03}. Let us give the main ingredients here.
\begin{proof}
Firstly, the domain $\mathcal D$ is chosen such that $D(f,g)(0^+)=D(f,g)(0^-)$ for $f,g \in \mathcal D$. So by Lemmas \ref{lem:Lfg <-> Dfg} and \ref{lem:D(f,g)(infty)} $(L,\mathcal D)$ is symmetric with respect to $\langle\cdot,\cdot\rangle$.

Secondly, if $f\in \mathcal D$ has support at only one point $x \in \R_q$, then $\langle Lf,g \rangle = \langle f, Lg\rangle$ for any $g \in F(\R_q)$. This holds in particular for $g \in \mathcal D^*$, where $(L^*,\mathcal D^*)$ is the adjoint of $(L,\mathcal D)$, so $(Lg)(x) = (L^*g)(x)$. We conclude that $L^*$ is the second-order $q$-difference operator $L$ restricted to $\mathcal D^*$.

Thirdly, for $f \in \mathcal D$ and $g \in \mathcal D^*$ we have $\langle Lf,g\rangle=\langle f,L^*g\rangle$, so from Lemmas \ref{lem:Lfg <-> Dfg} and \ref{lem:D(f,g)(infty)} we find $D(f,\overline g)(0^-) = D(f,\overline g)(0^+)$. This implies $g(0^-)=g(0^+)$ and $g'(0-)=g'(0^+)$, so $g \in \mathcal D$, hence $\mathcal D^* \subset \mathcal D$ and then $(L,\mathcal D)$ is self-adjoint.
\end{proof}

\subsection{Eigenfunctions} \label{ssec:eigenfunctions}
Our next goal is to determine eigenfunctions (in the algebraic sense) of the $q$-difference operator $L$. We will need the following result.
\begin{lemma} \label{lem:Vmu}
For $\mu \in \C$ we define
\[
V_\mu = \left\{ f \in F(\R_q) \mid Lf = (\mu-1)f, \ f(0^+)=f(0^-),\ f'(0^+)=f'(0^-) \right\}
\]
and
\[
V_\mu^\pm = \{ f \in F(\R_q^\pm) \mid Lf = (\mu-1)f \}.
\]
\begin{enumerate}[(i)]
\item For $f,g \in V_\mu^\pm$ the Casorati determinant $D(f,g) \in F(\R_q^\pm)$ is constant.
\item For $f,g \in V_\mu$, the Casorati determinant $D(f,g) \in F(\R_q)$ is constant.
\item We have $\dim(V_\mu) \leq 2$. Furthermore, in case $\dim(V_\mu)=2$, the restriction operator $V_\mu \to V_\mu^+$ is a bijection, and similarly for the restriction operator $V_\mu \to V_\mu^-$.
\end{enumerate}
\end{lemma}
For the proof, see \cite[Lemma 3.1]{Groen}. The main purpose of this lemma is to relate eigenfunctions of $L$ on $\R_q^+$ in a canonical way to eigenfunctions on $\R_q^-$ with the same eigenvalue.\\

We will now determine explicit eigenfunctions of $L$. We define
\begin{equation} \label{eq:psi}
\begin{split}
\psi_\ga(x)=\psi_\ga(x;c,d;q) &= (cq\ga x,dx;q)_\infty \rphis{1}{1}{q\ga}{cq\ga x}{q,\frac{cq}{d}}, \\ \psi_\ga^\dag(x)=\psi_\ga^\dag(x;c,d;q) &= (dq\ga x,cx;q)_\infty \rphis{1}{1}{q\ga}{dq\ga x}{q,\frac{dq}{c}}.
\end{split}
\end{equation}
Observe that $\psi_\ga^\dag$ is obtained from $\psi_\ga$ by interchanging $c$ and $d$. Furthermore, $\psi_\ga(x)$ and $\psi_\ga^\dag(x)$ are both entire functions in $x$ and in $\ga$. We also define
\begin{equation} \label{eq:Phi}
\Phi_\ga(x)=\Phi_\ga(x;c,d;q) = \frac{(cx;q)_\infty }{(q/dx,c\ga x;q)_\infty} \rphis{1}{1}{q/cx}{q/c\ga x}{q, \frac{q}{d\ga x}}.
\end{equation}
Other expressions for $\psi_\ga$, $\psi_\ga^\dag$, $\Phi_\ga$ can be obtained from Lemma \ref{lem:1phi1-transforms}. An expression for $\Phi_\ga$ that will be useful, is obtained from applying the transformation formula \eqref{eq:transform1},
\begin{equation} \label{eq:Phi2}
\Phi_\ga(x)=\frac{ (cx,1/\ga;q)_\infty} {(q/dx;q)_\infty \te(c\ga x;q)} \rphis{2}{1}{q/cx,q/dx}{0}{q,\frac{1}{\ga}}, \qquad |\ga|>1.
\end{equation}
\begin{prop}
The functions $\psi_\ga, \psi_\ga^\dag$ and $\Phi_\ga$ are solutions of the eigenvalue equation $Lf = (\ga -1) f$.
\end{prop}
\begin{proof}
This follows from $q$-contiguous relations. We write $\phi(a,b) = {}_2\varphi_1(a,b;0;q,z)$, then
\[
\phi(aq,b) - \phi(a,b)=az(1-b) \phi(aq,bq),
\]
see \cite[Exercise 1.9(i)]{GR}. Using the symmetry in $a$ and $b$ and replacing $(a,b)$ by $(a/q,b/q)$ gives us
\[
\phi(a/q,b) - \phi(a/q,b/q)=\tfrac{bz}{q}(1-a/q) \phi(a,b).
\]
Combining this with
\[
q\phi(a/q,b) + bz(1-a)\phi(aq,b) = [q-az(1-b/z)]\phi(a,b),
\]
see \cite[Exercise 1.10(iii)]{GR}, we obtain
\[
q \phi(a/q,b/q) + abz^2(1-a)(1-b) \phi(aq,bq) = [q-z(a+b)+abz(1+1/q)] \phi(a,b).
\]
Setting $a=q/cx$, $b=q/dx$ and $z=1/\ga$, we find after a straightforward calculation that $\Phi_\ga$ (as given by \eqref{eq:Phi2}) is an eigenfunction of $L$ for eigenvalue $\ga-1$. Since \eqref{eq:Phi} is the meromorphic continuation of \eqref{eq:Phi2} for $|\ga|\leq 1$, $\Phi_\ga$ is a solution of the eigenvalue equation $Lf=(\ga-1)f$ for generic values of $\ga$. (Later on we determine explicitly the poles of $\Phi_\ga$).

Next we write $\phi(b)= {}_1\varphi_1(a;b;q,z)$, then
\[
(q-b)(az-b) \phi(b/q) + [b(q-b) + bz - az(1+q)] \phi(b) - \tfrac{z(b-a)}{1-b} \phi(bq) =0,
\]
which is a limit case of \cite[Exercise 1.10(iv)]{GR}. Setting $a=q\ga$, $b=cq\ga x$ and $z=qc/d$, we find after a calculation that $\phi_\ga$ is an eigenfunction of $L$ for eigenvalue $\ga-1$. Interchanging $c$ and $d$ gives the result for $\phi_\ga^\dag$.
\end{proof}

The following two lemmas will be useful later on.
\begin{lemma} \label{lem:asymptotics}
The asymptotic behavior of $\psi_\ga(x)$ and $\Phi_\ga(x)$ for $|x|\to \infty$ in $\R_q$ is given by
\begin{gather*}
\psi_\ga(z_\pm q^k) = \te(cq\ga z_\pm,dz_\pm ;q)(cdq\ga z_\pm^2)^{-k} q^{-k(k-1)}\Big(1+\mathcal O(q^{-k})\Big), \\
\Phi_\ga(z_\pm q^k) = \frac{\te(cz_\pm;q)}{\te(c\ga z_\pm;q)} \ga^k \Big(1+\mathcal O(q^{-k})\Big),
\end{gather*}
as $k \to -\infty$.
\end{lemma}
\begin{proof}
This is a straightforward calculation using the $\te$-product identity \eqref{eq:te-product} and the explicit expressions \eqref{eq:psi}, \eqref{eq:Phi} for $\psi_\ga$ and $\Phi_\ga$
\end{proof}
\begin{lemma} \label{lem:expansion Phi psi}
For $k \in \Z$ and $\ga \in \C\setminus\{0\}$,
\[
(q\ga;q)_\infty \te(dz_\pm,c\ga z_\pm;q) \Phi_\ga(z_\pm q^k) = c_\pm(\ga) \psi_\ga(z_\pm q^k) + c^\dag_\pm(\ga) \psi_\ga^\dag(z_\pm q^k),
\]
where
\begin{equation} \label{eq:c-function}
c_\pm(\ga) = \frac{ \te(cz_\pm, d\ga z_\pm;q) }{\te(d/c;q) }, \qquad c_\pm^\dag(\ga) = \frac{ \te(dz_\pm, c\ga z_\pm;q) }{\te(c/d;q)}.
\end{equation}
\end{lemma}
\begin{proof}
This follows from the three-term transformation formula \eqref{eq:transform5} with parameters $(a,b,z) = (q/cx,q/dx,1/\ga)$ and with transformation \eqref{eq:transform4} applied to the two $_1\varphi_1$-functions on the right hand side, and definitions \eqref{eq:psi} and \eqref{eq:Phi} for $\psi_\ga$, $\psi_\ga^\dag$ and $\Phi_\ga$.
\end{proof}
We define for $x \in \R_q^+$ and $\ga \neq 0$,

\begin{equation} \label{eq:phi+}
\phi_\ga^+(x) =(q\ga;q)_\infty \te(dz_+,c\ga z_+;q) \Phi_\ga(x).
\end{equation}
Note the $\phi_\ga^+$ is a solution of the eigenvalue equation $Lf = (\ga-1)f$ on $\R_q^+$, so by Lemma \ref{lem:Vmu} it has (for generic values of $\ga$) an extension to $V_\ga$, which we will also denote by $\phi_\ga^+$. Later on we show that an explicit expression for this extension is already given by Lemma \ref{lem:expansion Phi psi}. Note also that by Lemma \ref{lem:expansion Phi psi} the function $\phi_\ga^+$ is symmetric in $c$ and $d$, so $\phi_\ga^+$ is real-valued for $\ga \in \R$. Similarly, we define for $x \in \R_q^-$,
\begin{equation} \label{eq:phi-}
\phi_\ga^-(x) = (q\ga;q)_\infty \te(dz_-,c\ga z_-;q) \Phi_\ga(x).
\end{equation}
This function also has an extension to $V_\ga$, and it is also symmetric in $c$ and $d$. Observe that $\phi_\ga^-$ is obtained from $\phi_\ga^+$ by replacing $z_+$ by $z_-$.

\begin{lemma} \label{lem:Casorati dets}
We have
\[
\begin{split}
D(\psi_\ga, \phi^+_\ga) & =-\frac{1-q}{c} (q\ga;q)_\infty \te(dz_+,c\ga z_+ ;q),  \\
D(\psi_\ga^\dag, \phi^+_\ga) &=-\frac{1-q}{d} (q\ga;q)_\infty\te(cz_+,d\ga z_+ ;q), \\
D(\psi_\ga^\dag,\psi_\ga) &= -\frac{1-q}{d} (q\ga;q)_\infty\te(d/c;q).
\end{split}
\]
In particular, for $\ga \in \C \setminus q^{-\N-1}$ the functions $\psi_\ga$ and $\psi_\ga^\dag$ are linearly independent.
\end{lemma}
\begin{proof}
Since the Casorati determinant is constant on $\R_q$ by Lemma \ref{lem:Vmu}, it is enough to compute the Casorati determinant on $\R_q^+$, which in turn can be computed by letting $x \to \infty$ in $\R_q^+$.
First we compute $D(\psi_\ga, \phi^+_\ga)$ and $D(\psi_\ga^\dag, \phi^+_\ga)$.
With the asymptotic behavior of $v$, $\psi_\ga$ and $\Phi_\ga$, see \eqref{eq:v(infty)} and Lemma \ref{lem:asymptotics}, we find
\begin{gather*}
\lim_{k \to -\infty} \psi_\ga(z_+ q^k) \Phi_\ga(z_+ q^{k+1}) v(z_+ q^k) = -\frac{1-q}{c},\\
\lim_{k \to -\infty} \psi_\ga(z_+ q^{k+1}) \Phi_\ga(z_+ q^{k}) v(z_+ q^k) =0,
\end{gather*}
so that $D(\psi_\ga,\Phi_\ga) = -\frac{1-q}{c}$. From \eqref{eq:phi+} we now find the expression for $D(\psi_\ga,\phi^+_\ga)$. Furthermore, by interchanging $c$ and $d$ we also obtain the expressions for $D(\psi_\ga^\dag, \phi^+_\ga)$.

Finally, using Lemma \ref{lem:expansion Phi psi} and \eqref{eq:phi+} we can expand $\psi_\ga$ on $\R^+_q$ in terms of $\phi_\ga^+$ and $\psi_\ga^\dag$. This leads to
\[
D(\psi_\ga^\dag,\psi_\ga) = \frac{D(\psi_\ga^\dag,\phi_\ga^+)}{c_+(\ga)},
\]
which gives the expression for $D(\psi_\ga^\dag,\psi_\ga)$.
\end{proof}

We can now give a basis for $V_\ga$.
\begin{prop} \label{prop:basis for V}
The functions $\psi_\ga$ and $\psi_\ga^\dag$ are in $V_\ga$. Furthermore, for $\ga \in \C \setminus (q^{-\N-1}\cup\{0\})$, the set $\{\psi_\ga,\psi_\ga^\dag\}$ is a linear basis for $V_\ga$.
\end{prop}
\begin{proof}
First we show that $\psi_\ga$ is in $V_\ga$; for $\psi_\ga^\dag$ the proof is the same.

We have
\[
\lim_{x \to 0} \psi_\ga(x) = \rphis{1}{1}{q\ga}{0}{q,\frac{cq}{d}},
\]
so clearly $\psi_\ga(0^+)= \psi_\ga(0^-)$.

Next we define $f(x;c,d) = (cq\ga x;q)_\infty {}_1\varphi_1(q\ga;cq\ga x;q,cq/d)$, then $\psi_\ga(x) = (dx;q)_\infty f(x;c,d)$. A straightforward calculation gives
\[
(D_qf)(x;c,d) = \frac{-cq\ga}{1-q} f(x;cq,d),
\]
and then
\[
\lim_{x \to 0} (D_q f)(x;c,d) = \frac{-cq\ga}{1-q} \rphis{1}{1}{q\ga}{0}{q,\frac{cq^2}{d}}.
\]
We also have $D_q (dx;q)_\infty = \frac{-d}{1-q}(dqx;q)_\infty$, so that $\lim_{x \to 0} D_q (dx;q)_\infty = \frac{-d}{1-q}$. Now from the product rule, $D_q(f(x)g(x)) = (D_q f)(x) g(x) + f(qx) (D_qg)(x)$, we see that $\lim_{x \to 0} (D_q\psi_\ga)(x)$ exists, so clearly $(D_q\psi_\ga)(0^+)=(D_q\psi_\ga)(0^-)$, so that $\psi_\ga$ in $V_\ga$.

Finally, for $\ga \not\in \{0\} \cup q^{-\N-1}$, $\psi_\ga$ and $\psi_\ga^\dag$ are linearly independent by Lemma \ref{lem:Casorati dets}. Since $\dim V_\ga \leq 2$ by Lemma \ref{lem:Vmu}, the set $\{\psi_\ga,\psi_\ga^\dag\}$ is a linear basis for $V_\ga$.
\end{proof}

\begin{cor} \label{cor:c-expansion phi}
For $\ga \in \C \setminus \{0\}$ the functions $\phi_\ga^\pm \in V_\ga$ are given explicitly by
\[
\phi_\ga^\pm = c_\pm(\ga) \psi_\ga + c_\pm^\dag(\ga) \psi_\ga^\dag.
\]
\end{cor}
\begin{proof}
This follows from Lemma \ref{lem:expansion Phi psi} and Proposition \ref{prop:basis for V}.
\end{proof}
If $\ga \in q^{-\N-1}$, it follows from  Lemma \ref{lem:Casorati dets} that $\psi_\ga$ and $\psi_\ga^\dag$ are not linearly independent. In this case we have the following result.
\begin{lemma} \label{lem:phi=0}
For $n \in \N$ we have $\psi_{q^{-n-1}} = (c/d)^n\, \psi_{q^{-n-1}}^\dag$ and $\phi_{q^{-n-1}}^\pm=0$.
\end{lemma}
\begin{proof}
We use transformation formula \eqref{eq:transform5} with $(a,b,z) = (q^{-n},q/dx,cx)$. The second term on the right hand side vanishes because of the factor $(q^{-n};q)_\infty$, and then we find after using the $\te$-product identity \eqref{eq:te-product},
\[
\begin{split}
\psi_{q^{-n-1}}(x) &= (cxq^{-n},dx;q)_\infty \rphis{1}{1}{q^{-n}}{cxq^{-n}}{q,\frac{cq}{d}}\\
 & =  (dxq^{-n},cx;q)_\infty \left(\frac{c}{d}\right)^n \rphis{1}{1}{q^{-n}}{dxq^{-n}}{q,\frac{dq}{c}},
\end{split}
\]
so that $d^n \psi_{q^{-n-1}}(x) = c^n \psi_{q^{-n-1}}^\dag(x)$. From the expansion in Corollary \ref{cor:c-expansion phi} and the $\te$-product identity \eqref{eq:te-product}, we now find $\phi_{q^{-n-1}}^\pm(x)=0$ for all $x \in \R_q$.
\end{proof}
We can now also compute the Casorati determinant $D(\phi_\ga^-, \phi_\ga^+)$.
\begin{lemma} \label{lem:Casorati phi+-}
For $\ga \in \C\setminus\{0\}$,
\[
D(\phi_\ga^+, \phi_\ga^-) = -(1-q)z_+\ga (q\ga;q)_\infty \te(z_-/z_+,1/\ga,cd\ga z_- z_+;q).
\]
In particular, $\phi_\ga^+$ and $\phi_\ga^-$ are linearly independent for $\ga \in \C\setminus\big(q^\Z \cup (1/cdz_-z_+)q^\Z \cup \{0\} \big)$.
\end{lemma}
\begin{proof}
Using Corollary \ref{cor:c-expansion phi} and Lemma \ref{lem:Casorati dets} we find
\[
\begin{split}
D(\phi_\ga^-, \phi_\ga^+) &= c_-(\ga) D(\psi_\ga, \phi_\ga^+) + c_-^\dag(\ga) D(\psi_\ga^\dag, \phi_\ga^+) \\
& = - (1-q)\frac{(q\ga;q)_\infty}{c\, \te(d/c;q)} \Big( \te(cz_-, d\ga z_-,dz_+,c\ga z_+;q) - \te(dz_-, c\ga z_-,cz_+,d\ga z_+;q)\Big).
\end{split}
\]
The result now follows from the fundamental $\te$-function identity, see \cite[Exercise 2.16(i)]{GR},
\[
\te(xv,x/v,yw,y/w)-\te(xw,x/w,yv,y/v) = \frac{y}{v} \te(xy,x/y,vw,v/w),
\]
with
\[
x=z_-\sqrt{cd \ga}, \quad y = z_+ \sqrt{cd\ga}, \quad v=\sqrt{c/d\ga}, \quad w = \sqrt{d/c\ga},
\]
where $\sqrt{\cdot}$ denotes the principal branch of the square root.
\end{proof}
In case $\phi_\ga^+$ and $\phi_\ga^-$ are not linearly independent, $\phi_\ga^+$ can explicitly be expressed as a multiple of $\phi_\ga^-$. Recall from Lemma \ref{lem:phi=0} that $\phi_\ga^+ = \phi_\ga^-=0$ for $\ga \in q^{-\N-1}$. In case $\ga \in q^\N$ we can express $\phi_\ga^\pm$ as a multiple of an Al-Salam--Carlitz polynomial of type II, see \eqref{eq:ASCIIpol}.
\begin{lemma} \label{lem:phi+=phi-}
Let $x \in \R_q$.
\begin{enumerate}[(i)]
\item For $n \in \N$, we have
\[
\begin{split}
\phi_{q^n}^+(x)& = \left(\frac{z_-}{z_+}\right)^n \frac{\te(cz_+,dz_+;q)}{\te(cz_-,dz_-;q)} \phi_{q^n}^-(x)\\
 &= (cz_+)^{-n} (q^{n+1};q)_\infty \te(cz_+,dz_+;q) V_n^{(c/d)}(cx;q).
\end{split}
\]
\item For $n \in \Z$, we have $\phi_{q^{n+1}/cdz_-z_+}^+(x) = (z_-/z_+)^n \phi_{q^{n+1}/cdz_-z_+}^-(x)$.
\end{enumerate}
\end{lemma}
\begin{proof}
The first statement follows from applying transformation \eqref{eq:transform6} to the $_1\varphi_1$-series \eqref{eq:Phi} and the $\te$-product identity \eqref{eq:te-product}; for $x \in \R_q^+$,
\[
\begin{split}
\phi_{q^n}^+(x) & = \frac{ (cx;q)_n (q^{n+1};q)_\infty \te(dz_+,cz_+q^n;q) }{ (q/dx;q)_\infty} \rphis{1}{1}{q/cx}{q^{1-n}/cx}{q,\frac{q^{1-n}}{dx}}\\
& = (-1)^n q^{-\frac12n(n-1)} (dz_+)^{-n} (q^{n+1};q)_\infty \te(cz_+,dz_+;q) \rphis{2}{0}{q^{-n},cx}{-}{q,\frac{dq^n}{c}}.
\end{split}
\]
This is a polynomial in $x$, hence it lies in $V_\ga$, so the same expression for $\phi_{q^n}^+(x)$ is valid for $x \in \R_q^-$. Furthermore, it is clearly a multiple of \eqref{eq:ASCIIpol} with $(k,a,z)=(n,c/d,cx)$. Dividing the expression by $(z_+)^{-n} \te(cz_+,dz_+;q)$ makes it independent of $z_-$ and $z_+$, so that $(z_+)^{n} \te(cz_+,dz_+;q)^{-1}  \phi_{q^n}^+ = (z_-)^n \te(cz_-,dz_-;q)^{-1}\phi_{q^n}^-$.

For the second statement, observe that the explicit expression \eqref{eq:c-function} for $c_\pm$ gives
\[
c_+\big( q^{n+1}/cdz_-z_+ \big) = \left( -\frac{ cz_- }{q} \right)^n q^{-\frac12 n(n-1)} \frac{ \te(cz_-,cz_+;q) }{ \te(d/c;q)} = \left(\frac{z_-}{z_+} \right)^n c_-\big( q^{n+1}/cdz_-z_+ \big),
\]
by the $\te$-product identity \eqref{eq:te-product}. Similarly, $c_+^\dag\big( q^{n+1}/cdz_-z_+ \big) = (z_-/z_+)^n c_-^\dag\big( q^{n+1}/cdz_-z_+ \big)$. Then the expansion from Corollary \ref{cor:c-expansion phi} gives the result.
\end{proof}

\subsection{Spectral decomposition} \label{sect:spectraldecomposition}
We can now calculate the resolvent operator for the self-adjoint operator $L$ explicitly. We define the Green kernel $K_\ga:\R_q\times \R_q \to \C$ by
\[
K_\ga(x,y) =
\begin{cases}
\displaystyle \frac{\phi_\ga^-(x) \phi_\ga^+(y)}{D(\ga)}, & x \leq y,\\ \\
\displaystyle \frac{\phi_\ga^-(y) \phi_\ga^+(x)}{D(\ga)}, & x \geq y,
\end{cases}
\]
where $D(\ga) = D(\phi_\ga^+,\phi_\ga^-)$.
Observe that for $x,y \in \R_q$ we have $K_\ga(x,\cdot), K_\ga(\cdot,y) \in \mathcal H$. In order to determine the spectral decomposition for the self-adjoint operator $L$ we need to know the location of  the poles of $\ga \mapsto K_\ga(x,y)$.
\begin{lemma} \label{lem:poles}
For $x,y \in \R_q$, $\ga \mapsto K_\ga(x,y)$ has simple poles in $q^\N \cup (1/cdz_-z_+)q^\Z$, and is analytic on $\C \setminus \big(q^\N \cup (1/cdz_-z_+)q^\Z \cup \{0\}\big)$.
\end{lemma}
\begin{proof}
The poles of $K_\cdot(x,y)$ can come from the poles of $\ga\mapsto \phi_\ga^\pm$ and from zeros of $D(\phi_\ga^+,\phi_\ga^-)$. We use Corollary \ref{cor:c-expansion phi} to see that $\ga \to \phi^\pm_\ga(x)$ is analytic on $\C\setminus\{0\}$. Indeed, from \eqref{eq:psi} it follows that $\ga \mapsto \psi_\ga(x)$ and $\ga \mapsto \psi_\ga^\dag$ are entire functions, and the functions $\ga\mapsto c_\pm(\ga)$ and $\ga \mapsto c_\pm^\dag(\ga)$, see Lemma \ref{lem:expansion Phi psi} for the explicit expressions, are analytic on $\C\setminus\{0\}$.

The location of the zeros of $D(\phi_\ga^+,\phi_\ga^-)$ can be read off from the explicit expression in Lemma \ref{lem:Casorati phi+-}: simple zeros in $q^\N \cup (1/cdz_-z_+)q^\Z$ and double zeros in $q^{-\N-1}$. The double zeros are canceled by the zeros of $\ga\mapsto \phi^+_\ga(x) \phi^-_\ga(y)$, see Lemma \ref{lem:phi=0}, so $\ga\mapsto K_\ga(x,y)$ has simple poles at the simple zeros of $D(\phi_\ga^+,\phi_\ga^-)$.
\end{proof}

We use the Green kernel to describe the resolvent for $L$.
\begin{prop} \label{prop:resolvent}
For $\mu \in \C\setminus\R$ the resolvent $R(\mu)$ for $(L,\mathcal D)$ is given by
\[
R(\mu)f(y) = \big\langle f, \overline{K_{\mu+1}(\,\cdot\,,y)}\big\rangle, \qquad f \in \mathcal H,\ y \in \R_q.
\]
\end{prop}
\begin{proof}
The proof boils down to checking that $\big((L-\mu) R(\mu) f \big)(y)=f(y)$, see e.g.~\cite[Proposition 6.1]{KS03}.
\end{proof}

To determine the spectral measure $E$ for the self-adjoint operator $(L,\mathcal D)$ we use, see \cite[Theorem XII.2.10]{DS63},
\begin{equation} \label{eq:Stieltjes-Perron}
\langle E(a,b)f, g \rangle_{\mathcal H} = \lim_{\de \downarrow 0} \lim_{\eps \downarrow 0} \frac{1}{2\pi i} \int_{a + \de}^{b-\de} \Big( \langle R(\mu+i\eps) f,g \rangle_{\mathcal H} - \langle R(\mu-i\eps) f,g \rangle_{\mathcal H} \Big) d\mu,
\end{equation}
for $a<b$ and $f,g \in \mathcal H$.
\begin{thm} \label{thm:spectraldecomposition}
The self-adjoint operator $(L,\mathcal D)$ has discrete spectrum $\mathcal S-1$, with
\[
\mathcal S= q^\N  \cup (1/cdz_-z_+)q^\Z,
\]
and continuous spectrum $\{-1\}$. For $\ga \in \mathcal S$, let $(a,b)$ be an interval such that $(a,b) \cap (\mathcal S-1) = \{\ga-1\}$. Then, for $f,g \in \mathcal H$,
\[
\big\langle E(a,b)f,g\big\rangle = -\frac{1}{d(\ga)} \Res{\ga'=\ga} \frac{1}{D(\ga')} \langle f,\phi_{\ga}^+ \rangle \langle \phi_{\ga}^+,g\rangle,
\]
where
\[
d(\ga) =
\begin{cases}
\displaystyle \left(\frac{z_-}{z_+}\right)^n \frac{\te(cz_+,dz_+;q)}{\te(cz_-,dz_-;q)}, & \text{for}\ \ga = q^n,\ n \in \N,\\ \\
\displaystyle \left( \frac{z_-}{ z_+ } \right)^n, & \text{for}\ \ga = q^{n+1}/cdz_-z_+,\ n \in\Z.
\end{cases}
\]
\end{thm}
\begin{proof}
We have
\[
\begin{split}
\langle R(\ga-1) f,g \rangle &= \iint\limits_{\R_q \times \R_q} f(x) \overline{g(y)} K_\ga(x,y) w(x) w(y)\, d_qx\, d_qy \\
& = \iint\limits_{\substack{(x,y) \in \R_q \times \R_q \\ x \leq y }} \frac{ \phi_\ga^-(x) \phi_\ga^+(y) }{D(\ga)} \Big( f(x) \overline{g(y)} + f(y) \overline{g(x)} \Big) \Big(1-\frac12 \de_{x,y}\Big) w(x) w(y)\, d_qx\, d_qy,
\end{split}
\]
so from \eqref{eq:Stieltjes-Perron} and Lemma \ref{lem:poles} we see that the only values of $\ga \in \R$ that contribute to the spectral measure $E$ are the poles of $\ga \mapsto K_\ga(x,y)$, i.e.,~$\ga \in q^\N \cup (1/cdz_-z_+) q^\Z$. In this case we know that  $\phi_\ga^+= d(\ga) \phi_\ga^-$ where $d(\ga)$ is the explicit factor from Lemma \ref{lem:phi+=phi-}, so $\phi_\ga^+ \in \mathcal H$, which implies that $\phi_\ga^+$ is an eigenfunction of $L$ for eigenvalue $\ga-1$. We fix such a $\ga$ and we choose $a$ and $b$ such that $(a,b) \cap (\mathcal S-1) = \{\ga-1\}$, then by \eqref{eq:Stieltjes-Perron} we have $\langle E(a,b) f,g\rangle = \frac{1}{2\pi i} \int_{\mathcal C} \langle R(\mu)f,g\rangle d\mu$, where $\mathcal C$ is a clockwise oriented, rectifiable contour encircling $\ga-1$ once. From Cauchy's theorem we find
\begin{multline*}
\langle E(a,b) f,g\rangle =\\ -\frac{1}{d(\ga)} \Res{\ga'=\ga} \frac{1}{D(\ga')} \iint\limits_{\substack{(x,y) \in \R_q \times \R_q \\ x \leq y }} \phi_\ga^+(x) \phi_\ga^+(y) \Big( f(x) \overline{g(y)} + f(y) \overline{g(x)} \Big) \Big(1-\frac12 \de_{x,y}\Big) w(x) w(y)\, d_qx\, d_qy,
\end{multline*}
where the minus sign comes from the negative orientation of the contour $\mathcal C$. Then the expression for the spectral measure follows from symmetrizing the double $q$-integral.

It remains to show that $-1$ is in the continuous spectrum. First note that $-1$ is in the closure of $\mathcal S-1$, so $-1$ is either in the discrete spectrum or in the continuous spectrum. Suppose $f \in F(\R_q)$ satisfies $Lf=-f$, then the restriction of $f$ to $\R_q^+$ is a linear combination of $\psi_0$ and $\psi_0^\dag$. From \eqref{eq:psi} we see that $\psi_0(x) = (cq/d,dx;q)_\infty$ and $\psi_0^\dag =  (dq/c,cx;q)_\infty$. Using the $\te$-product identity and \eqref{eq:asymptotics w}, we obtain
\[
|(cx;q)_\infty|^2w(x)\, x = \mathcal O( |qd/c|^k ), \qquad x = z_+ q^k, \ k \to -\infty.
\]
From the conditions on $c$ and $d$ we know that $|d/c|<q^{-1}$, so that $(cx;q)_\infty \not\in \mathcal H$. Similary, $(dx;q)_\infty \not\in \mathcal H$. We conclude that $f \not\in \mathcal H$, so $-1$ is not an eigenvalue of $(L,\mathcal D)$.
\end{proof}

\subsection{Orthogonality relations} \label{ssec:orthogonality}
From Theorem \ref{thm:spectraldecomposition} we obtain orthogonality relations for the functions $\phi_\ga^+$.
\begin{prop} \label{prop:orthogonality phi+}
The set $\{ \phi_\ga^+ \}_{\ga \in \mathcal S}$ is an orthogonal basis for $\mathcal H$, with squared norm given by
\[
\| \phi_\ga^+ \|^2 = -d(\ga) \left( \Res{\ga'=\ga} \frac{1}{D(\ga')} \right)^{-1}.
\]
\end{prop}
\begin{proof}
Since $L$ is self-adjoint, eigenfunctions for different eigenvalues are pairwise orthogonal. Let $\ga \in \mathcal S$ and choose $a$ and $b$ such that  $(a,b)\cap (\mathcal S-1) = \{\ga-1\}$. Taking $f=g = \phi_\ga^+$  in Theorem \ref{thm:spectraldecomposition} gives
\[
\| \phi_\ga^+ \|^2 =- \frac{1}{d(\ga)} \Res{\ga'=\ga} \frac{1}{D(\ga')}\| \phi_\ga^+\|^4,
\]
from which the expression for the squared norm follows.

Next we define for $x\in \R_q$ the function $f_x \in \mathcal H$ by $f_x(y)=\de_{xy}/ (|x|w(x))$ for all $y \in \R_q$, then $f(x) = \langle f,f_x\rangle$ for all $f \in \mathcal H$. Now assume that for some $f \in \mathcal H$ we have $\langle f,\phi_\ga^+\rangle=0$ for all $\ga \in \mathcal S$. Then by Theorem \ref{thm:spectraldecomposition}
\[
f(x) = \langle f,f_x \rangle = \langle E(\R) f, f_x \rangle = \sum_{\ga \in \mathcal S} \langle f, \phi_\ga^+ \rangle \phi^+_\ga(x) \frac{-1}{d(\ga)} \Res{\ga'=\ga} \frac{1}{D(\ga')}=0,
\]
for all $x \in \R_q$. We conclude that $\{\phi_\ga^+\}_{\ga \in \mathcal S}$ is complete in $\mathcal H$.
\end{proof}

We reformulate the orthogonality relations for $\phi_\ga^+$ in terms of the following functions.
We define
\[
P_n(x) = (-c)^n q^{-\frac12 n(n-1)} \rphis{2}{0}{q^{-n},cx}{-}{q,\frac{dq^n}{c}},
\qquad n \in \N,
\]
as in the introduction, and
\[
Q_n(x) = (-d)^{n} q^{-\frac12n(n+1)} \frac{(cx;q)_\infty \te(dz_-,dz_+;q) }{ (q/dx, q^{n+1}x/dz_-z_+;q)_\infty} \rphis{1}{1}{q/cx}{dz_-z_+ q^{-n}/x}{q,\frac{cz_-z_+ q^{-n}}{x}}, \qquad n \in \Z.
\]
Recall that $P_n$ is symmetric in $c$ and $d$. Furthermore, $Q_n$ is symmetric in $c$ and $d$, and in $z_-$ and $z_+$. These functions are related to $\phi_\ga^+$ by
\begin{gather*}
\phi_{q^n}^+(x) = (cdz_+)^{-n} (q^{n+1};q)_\infty \te(cz_+,dz_+;q)\, P_n(x),\\
\phi_{q^{n+1}/cdz_-z_+}^+(x) = (z_-)^n (q^{n+2}/cdz_-z_+;q)_\infty \, Q_n(x),
\end{gather*}
see Lemma \ref{lem:phi+=phi-}.
\begin{thm} \label{thm:orthogonalityPQ}
The set $\{P_n\}_{n \in \N} \cup \{Q_n\}_{n \in \Z}$ is an orthogonal basis for $\mathcal H$. In particular, the following orthogonality relations hold:
\begin{align*}
\frac{1}{1-q}\int_{\R_q} P_m(x) P_n(x) \frac{1}{(cx,dx;q)_\infty} \,d_q x &= \de_{mn} (q;q)_n(cd)^{n} q^{-n^2} \frac{z_+\,(q;q)_\infty\te(z_-/z_+,cdz_-z_+;q)}{\te(cz_-,dz_-,cz_+,dz_+;q)},\\
\frac{1}{1-q}\int_{\R_q} Q_m(x) Q_n(x) \frac{1}{(cx,dx;q)_\infty} \,d_q x &= \de_{mn} \frac{  (cdz_-z_+q^{-n-1};q)_\infty }{(-z_-z_+)^{n} q^{\frac12 n(n+1)}}\,z_+ (q;q)_\infty^2 \te(z_-/z_+;q),\\
\frac{1}{1-q}\int_{\R_q} P_m(x) Q_n(x) \frac{1}{(cx,dx;q)_\infty} \,d_q x &=0.
\end{align*}
\end{thm}
\begin{proof}
This follows from Proposition \ref{prop:orthogonality phi+}. The explicit squared norms follow from a straightforward residue calculation using the explicit expression for $D(\ga)$ from Lemma \ref{lem:Casorati phi+-};
\begin{multline*}
\left( \Res{\ga=q^{n+1}/cdz_-z_+} \frac{1}{D(\ga)} \right)^{-1}=\\
(1-q)z_+ (-1)^{n+1} q^{-\frac12n(n+1)} (q,q,q^{n+2}/cdz_-z_+;q)_\infty \te(z_-/z_+, cdz_-z_+q^{-1-n}),
\end{multline*}
and
\[
\left( \Res{\ga=q^n} \frac{1}{D(\ga)} \right)^{-1}= (1-q)z_+ (-1)^{n+1} q^{\frac12n(n-1)}q^{-n^2} (q,q,q^{n+1};q)_\infty \te(z_-/z_+,cdz_-z_+q^n).
\qedhere
\]
\end{proof}
Observe that (the proof of) Theorem \ref{thm:orthogonalityPQ} gives an alternative proof of Theorem \ref{thm:orthogonalitypolynomials} and the evaluation formula \ref{eq:beta-integral}.
\begin{remark} \label{rem:qHermite}
The polynomials $\widetilde h_n(x;q) = i^{-n} V_n^{(-1)}(ix;q)$ are known as discrete $q$-Hermite polynomials of type II, see \cite{KLS}. So for $c= i a$ with $a \in \R$ (in this case $d=- ia$), Theorem \ref{thm:orthogonalityPQ} gives solutions for the indeterminate moment problem corresponding to the discrete $q$-Hermite II polynomials, as well as a set of functions that complement them to an orthogonal basis of the corresponding Hilbert space. If $z_-=- z_+$ the orthogonality relations from Theorem \ref{thm:orthogonalityPQ} correspond to the orthogonality relations given in \cite{KLS}. In this case the solutions for the moment problem can be obtained from Ramanujan's $_1\psi_1$-summation formula, see \cite{B}.
\end{remark}

\section{Orthogonality relations on $[1/c,\infty)$} \label{sect:classicorthogonality}
In this section we obtain orthogonality relations from spectral analysis of the Al-Salam--Carlitz II second order $q$-difference operator acting on a Hilbert space $H$ which is essentially $\mathcal H$ from Section \ref{sect:orthogonality on R} with $z_-=0$ and $z_+=1/c$. The calculations are similar to the ones in Section \ref{sect:orthogonality on R}, which is why most of the details are omitted.
\\

We assume in this section $c>0$ and $d>0$, and we set $z_+=\frac1c$ and $z_-=0$. We consider $L$ as a bounded operator on the Hilbert space $H$ consisting of complex-valued functions $f$ on the $q$-interval $I_q = (1/c)q^{-\N}$ with inner product
\[
\langle f,g\rangle_{H} = \int_{I_q} f(x)\overline{g(x)} W(x) \,d_q x.
\]
Here
\[
\int_{I_q} f(x)\, d_q x = (1-q) \sum_{k=0}^\infty f(q^{-k}/c) \frac{1}{cq^{k}},
\]
and $W$ is a normalized version of the weight function $w$ \eqref{eq:def w} restricted to $I_q$:
\[
W(q^{-k}/c) = \te(cz_+,dz_+;q) w(z_+ q^{-k}) \big|_{z_+=1/c} = \left(\frac{c}{d}\right)^{k} q^{k(k+1)}(cq^{k+1}/d,q^{k+1};q)_\infty, \qquad k \in \N.
\]
Note that $\te(cz_+,dz_+;q) w(x) \big|_{z_+=1/c}=0$ for $x \in (1/c)q^{\N+1}$. The evaluation of the corresponding $q$-beta integral can be obtained as a limit case of \eqref{eq:beta-integral};
\[
\int_{I_q} W(x) \,d_q x = \frac{1-q}{c}(q;q)_\infty.
\]
This can also be written as a $_0\varphi_1$-summation formula, which is a limit case of the $_1\varphi_1$-summation formula \cite[(II.8)]{GR}.

The second order $q$-difference operator $L:F(I_q) \to F(I_q)$ is defined in this case as follows. For $x \in I_q\setminus\{1/c\} =  (1/c) q^{-\N-1}$, $(Lf)(x)$ is still defined by \eqref{eq:L}, and we define
\begin{equation} \label{eq:boundary condition}
(Lf)(1/c) = B(1/c)\Big(f(1/cq)-f(1/c)\Big).
\end{equation}
Since $A(1/c) =0$, this is the natural definition for $L$. We denote the Casorati determinant corresponding to $L$ by $\widetilde D$, i.e., for $f,g \in F(I_q\setminus\{1/c\})$,
\[
\widetilde D(f,g)(q^{-k}/c) = \te(cz_+,dz_+;q) D(f,g)(z_+q^{-k})\big|_{z_+=1/c}, \qquad k \in \N_{\geq 1},
\]
where $D$ is defined by \eqref{eq:Casorati determinant}.

We also need to normalize the eigenfunctions in a different way. We define for $k \in \N$,
\[
\begin{split}
\Psi_\ga(q^{-k}/c) &=  \left.\frac{\psi_\ga(z_+ q^{-k})}{\te(dz_+;q)} \right|_{z_+=1/c}= q^{-\frac12k(k+1)} \left( -\frac{d}{c} \right)^k \frac{(\ga q^{1-k};q)_\infty}{(cq^{1+k}/d;q)_\infty } \rphis{1}{1}{q\ga}{\ga q^{1-k}}{q, \frac{cq}{d}},\\
\Psi_\ga^\dag(q^{-k}/c) &=  \left.\frac{\psi_\ga^\dag(z_+ q^{-k})}{\te(cz_+;q)} \right|_{z_+=1/c}= (-1)^k q^{-\frac12k(k+1)}  \frac{(d\ga q^{1-k}/c;q)_\infty}{(q^{1+k};q)_\infty } \rphis{1}{1}{q\ga}{d\ga q^{1-k}/c}{q, \frac{dq}{c}},
\end{split}
\]
and
\[
\phi_\ga(q^{-k}/c) = \left.\frac{\phi_\ga^+(z_+q^{-k})}{\te(cz_+,dz_+;q)}\right|_{z_+=1/c}.
\]
The expansion from Corollary \ref{cor:c-expansion phi} shows that
\begin{equation} \label{eq:Cexpansion}
\phi_\ga(x) = C(\ga) \Psi_\ga(x) + C^\dag(\ga) \Psi_\ga^\dag(x),\qquad x \in I_q,
\end{equation}
with
\[
C(\ga) = \frac{ \te(d\ga/c;q)}{\te(d/c;q)}, \qquad C^\dag(\ga) = \frac{\te(1/\ga;q)}{\te(c/d;q)}.
\]
Note that the functions $\Psi_\ga$ and $\Psi_\ga^\dag$ are eigenfunctions of $L$ on $I_q\setminus \{1/c\}$, but it should still be checked whether they satisfy the boundary condition $Lf(1/c) = (\ga-1)f(1/c)$, see \eqref{eq:boundary condition}. Using \eqref{eq:transform6} and \eqref{eq:te-product} we find
\[
\Psi_\ga(q^{-k}/c) = (q\ga;q)_\infty \rphis{2}{0}{q^{-k},1/\ga}{-}{q,\frac{d\ga}{c}}, \qquad k \in \N.
\]
With this expression it is easy to verify that $(L\Psi_\ga)(1/c)=(\ga-1)\Psi_\ga(1/c)$, so $\Psi_\ga$ does satisfy the boundary condition.

The operator $L$ is a self-adjoint operator on $H$. We calculate the spectral measure $E$ for $L$ in the same way as in Section \ref{sect:spectraldecomposition}.
In this case we define the Green kernel by
\[
K_\ga(x,y) =
\begin{cases}
\displaystyle \frac{\Psi_\ga(x) \phi_\ga(y)}{\widetilde D(\ga)}, & x \leq y,\\ \\
\displaystyle \frac{\Psi_\ga(y) \phi_\ga(x)}{\widetilde D(\ga)}, & x \geq y,
\end{cases}
\]
for $x,y \in I_q$ and where $\widetilde D(\ga) = \widetilde D(\phi_\ga,\Psi_\ga)$. The resolvent operator for $L$ is then given in terms of the Green kernel in the same way as in Proposition \ref{prop:resolvent}.

From Lemma \ref{lem:Casorati dets} we obtain
\[
\widetilde D(\ga) = \left.\frac{D(\phi_\ga^+,\psi_\ga)}{\te(dz_+;q)}\right|_{z_+=1/c} = -\frac{(1-q) \ga}{c} (q\ga;q)_\infty \te(1/\ga;q),
\]
and from this expression we see that $\phi_\ga$ and $\Psi_\ga$ are linear independent solutions of $Lf=(\ga-1)f$ for $\ga \in \C \setminus\big(q^\Z \cup \{0\}\big)$. Note that $\phi_\ga = \Psi_\ga = 0$ for $\ga \in q^{-\N-1}$, which follows from Lemma \ref{lem:phi=0} and linear dependence. Furthermore, for $\ga \in q^\N$ we have $\phi_\ga = C(\ga) \Psi_\ga$ by \eqref{eq:Cexpansion}, since $C^\dag(\ga)=0$ in this case. Since $\ga \mapsto \phi_\ga(x)$ and $\ga \mapsto \Psi_\ga(x)$ are entire functions, the function $\ga \mapsto K_\ga(x,y)$, $x,y \in I_q$, has simple poles in $q^\N$ and is analytic on $\C \setminus \big(q^\N \cup \{0\}\big)$. In the same way as in Theorem \ref{thm:spectraldecomposition} we can now calculate the spectral measure $E$, and as in Proposition \ref{prop:orthogonality phi+} this leads to orthogonality relations in $H$ for $\phi_{q^n}$. Using Lemma \ref{lem:phi+=phi-} we have
\[
\phi_{q^n}(x) = (q^{n+1};q)_\infty V_n^{(c/d)}(cx;q), \qquad x \in I_q,
\]
so we obtain orthogonality relations for Al-Salam--Carlitz II polynomials.
\begin{thm}\* \label{thm:orthogonality2}
\begin{enumerate}[(i)]
\item The self-adjoint operator $L$ on $H$ has discrete spectrum $q^\N-1$ and continuous spectrum $\{-1\}$. For $\ga \in q^\N$, let $(a,b)$ be an interval such that $(a,b)\cap (q^\N-1) = \{\ga-1\}$, then the spectral measure $E$ of $L$ is given by
\[
\langle E(a,b)f,g\rangle_H = -\frac{1}{C(\ga)} \Res{\ga'=\ga} \frac{1}{\widetilde D(\ga')}\langle f,\phi_\ga \rangle_H \langle \phi_\ga,g \rangle_H, \qquad f,g \in H.
\]
\item The set $\{V_n^{(c/d)}(c\,\cdot\,;q)\}_{n \in \N}$ is an orthogonal basis for $H$ with squared norm
    \[
    \| V_n^{(c/d)}(c\,\cdot\,;q) \|^2_H =\frac{1-q}{c}(q;q)_\infty \left(\frac{c}{d}\right)^n q^{-n^2} (q;q)_n.
    \]
\end{enumerate}
\end{thm}
\begin{remark}
The (Hamburger) moment problem corresponding to the Al-Salam--Carlitz II polynomials $V_n^{(a)}$ is indeterminate if $q<a < 1/q$. So in case $q<c/d<1/q$, Theorem \ref{thm:orthogonality2} provides an $N$-extremal solution for the Al-Salam--Carlitz II moment problem, since $\{V_n^{(c/d)}(c\,\cdot\,;q)\}_{n \in \N}$ forms a basis for $H$. This solution is first obtained by Al-Salam and Carlitz in \cite{ASC}. It is proved by Chihara in \cite{Ch}, and later by Berg and Valent in \cite{BV}, that the solution is $N$-extremal.
\end{remark}

\end{document}